\documentclass[11pt,a4paper]{article}
\usepackage{amsmath,amsfonts,amsthm,amssymb, mathtools}
\usepackage{mathrsfs, graphicx,color,latexsym, tikz, calc,
}
\usepackage{tikz-cd}
\usepackage{graphicx}
\usepackage{epstopdf}
\usepackage{enumerate}
\usepackage{caption}
\usepackage{stmaryrd}
\usepackage{hyperref}

\usetikzlibrary{shadows}
\usetikzlibrary{patterns,arrows,decorations.pathreplacing}
\usepackage{ulem}

\numberwithin{equation}{section}
\newtheorem{theorem}{Theorem}[section]
\newtheorem{lemma}[theorem]{Lemma}
\newtheorem{definition}[theorem]{Definition}
\newtheorem{problem}[theorem]{Problem}

\newtheorem{corollary}[theorem]{Corollary}

\allowdisplaybreaks[4]

\date{}
\title{\bf Bollob\'{a}s-type inequalities for subspaces via weight invariance\Large \footnote{L. Feng was supported by
the NSFC (Nos. 12271527 and 12471022).
  T. Wu  was supported by NSF of Qinghai Province (No. 2025-ZJ-902T), and   NSFC (No. 12261071). E-mail addresses: \url{liuzymath@163.com}(Z. Liu), \url{fenglh@163.com} (L. Feng), \url{mathtzwu@163.com} (T. Wu).
}}
\author{
{\small Zhiyi Liu$^a$, \ \  Lihua Feng$^a$,  \ \ Tingzeng Wu$^{b,c}$
  }\\[2mm]
\small $^a$School of Mathematics and Statistics, HNP-LAMA, Central South University\\
 \small Changsha, Hunan, 410083, China\\
 \small $^b$School of Mathematics and Statistics,
 Qinghai Minzu University\\
  \small Xining, Qinghai, 810007,  China\\
 \small $^c$Qinghai Institute of Applied Mathematics, Xining, Qinghai, 810007,   China\\
}

\begin{document}
\maketitle
\begin{abstract}

Let $V$ be an $n$-dimension real vector space with a direct sum decomposition $V = V_1 \oplus \cdots \oplus V_r$. Let  $\mathcal{P} = \{(A_i, B_i) : i \in [m]\}$
be  a skew Bollob\'as system of subspaces of $V$ such that each $i\in [m]$, 
 $ A_i = \bigoplus_{k=1}^r (A_i \cap V_k)$ and $ B_i = \bigoplus_{k=1}^r (B_i \cap V_k)$.
 We prove that
$$\sum_{i=1}^{m} \prod_{k=1}^{r} \left[ \binom{a_{i,k} + b_{i,k}}{a_{i,k}} (1 + a_{i,k} + b_{i,k})^{-1} \right] \leq 1,$$
where $a_{i,k} = \dim(A_i \cap V_k)$ and $b_{i,k} = \dim(B_i \cap V_k)$. This extends a recent result of Yue  from set systems to finite dimensional subspaces.
We then consider  Tuza's theorem on weak Bollob\'as system  for $d$-tuples.
We give an alternative proof of the original set version of Tuza, and also establish its vector space analogue.
Precisely,  let $\mathcal{P} = \{(A_i^{(1)}, \ldots, A_i^{(d)}) : i \in [m]\}$ be a skew Bollob\'as system  of $d$-tuples of subspaces of finite dimensional space $V$ with  $a^{(\ell)}_i=\dim (A_i^{(\ell)})$.  Then, for any positive real numbers $p_1, \ldots, p_d$ satisfying $p_1 + \cdots + p_d = 1$,  we prove that
$
\sum_{i=1}^{m} \prod_{\ell=1}^{d} p_{\ell}^{a_i^{(\ell)}} \leq 1.
$
\end{abstract}


\section{Introduction}\label{se1}
As one of the  central and fast growing topics among extremal combinatorics in the last couple of years, extremal set theory seeks to determine the maximum or minimum size of a family of sets subject to specific constraints. One of the most influential results in this area  is the Bollob\'as's celebrated Two Families Theorem (or set pair inequality) \cite{B65} proved   in 1965, which gives a fundamental inequality for families of set pairs with prescribed intersection constraints.
%This theorem, notable for its non-uniform nature and its proof via a now-classic combinatorial shuffling or "wedge" argument, has become a central tool in extremal set theory. Its power and generality have led to numerous generalizations and a wide array of applications, ranging from bounding the size of codes to analyzing graph coloring problems and the structure of partially ordered sets.
We begin by introducing the concept of a
Bollob\'{a}s system.

\begin{definition}\label{de1}
Let $[n]=\{1,\ldots, n\}$ and $\mathcal{P}=\{(A_i,B_i):i\in[m]\}$ be a collection of $m$ pairs of subsets of $[n]$. Then $\mathcal{P}$ is called a \textit{Bollob\'{a}s system} if
\begin{itemize}
    \item[(i)]
     for any $\ i\in[m]$,  $A_i\cap B_i=\emptyset$; and

     \item[(ii)]  for any $\ i\ne j\in[m]$, $A_i\cap B_j\neq \emptyset$.
     \end{itemize}
\end{definition}
Bollob\'{a}s proved the following landmark result, with a striking feature  that
the bound does not depends on the size of the ground set.

\begin{theorem}[Bollob\'{a}s \cite{B65}]\label{B65}
  Suppose that $A_i,B_i\subseteq[n],~i\in[m]$, and  $\mathcal{P}=\{(A_i,B_i):i\in[m]\}$ is a Bollob\'{a}s system with $|A_i|=a_i$ and $|B_i|=b_i$ for every $i$. Then
    $$
  \sum_{i=1}^{m}\frac{1}{\binom{a_i+b_i}{a_i}}\le 1.
    $$
\end{theorem}

For the uniform case with  $|A_i|=a$ and $|B_i|=b$ for all $i\in [m]$, Theorem \ref{B65} implies  the sharp bound $m\leq \binom{a+b}{a}$, which was proved by Jaeger and Payan \cite{J71} in 1971, and by Katona \cite{K74} in 1974 independently. Theorem \ref{B65} has a wide range of applications, such as covering problems for graphs \cite{H64, O77}, counting cross-intersecting families \cite{F18},  and transversals of hypergraphs \cite{T84, T85, T96}.
 Theorem \ref{B65} also has been generalized in a number of different directions in the literature.
 A natural variant arises when the intersection condition is required only in one direction in the following manner.
\begin{definition}\label{def0103}
    Let $\mathcal{P}=\{(A_i,B_i):i\in[m]\}$ be a collection of $m$ pairs of subsets of $[n]$. Then  $\mathcal{P}$ is called a \textit{skew Bollob\'{a}s system} if
    \begin{itemize}
    \item[(i)] for any $\ i\in[m]$,   $A_i\cap B_i=\emptyset$; and

     \item[(ii)]   for any $\ i<j\in[m]$,  $A_i\cap B_j \neq \emptyset$.
     \end{itemize}
\end{definition}

For the uniform case with  $|A_i|=a$ and $|B_i|=b$ for all $i\in [m]$, Lov\'asz \cite{L77}, and independently   Frankl \cite{F82},  proved that
 $$
 m\leq \binom{a+b}{a}
 $$
 holds under the weaker conditions in Definition \ref{def0103}. Alon \cite{A851} further extended this result  to the setting of $r$-partitions.

\begin{theorem}[Alon \cite{A851}]
    Let $[n]$ be the disjoint union of some sets $X_1,\ldots,X_r$. Suppose that $A_i,B_i\subseteq[n],\ i\in[m]$, and $\mathcal{P}=\{(A_i,B_i):i\in[m]\}$ is a skew Bollob\'{a}s system satisfying that
    $$
   |A_i\cap X_k|=a_k,~|B_i\cap X_k|=b_k,~\forall ~i\in[m], k\in[r].
    $$
    Then
    $$
    m\le\prod_{k=1}^{r}\binom{a_k+b_k}{a_k}.
    $$
\end{theorem}

For the nonuniform case with $a_k$ (resp. $b_k$) being not a constant, however, things become a bit complicate. The natural analogue of Theorem \ref{B65} fails for skew Bollob\'{a}s systems without further assumptions. It is exhilarating that, Scott and Wilmer \cite{S21}  achieved a breakthrough and showed a monotone version holds under additional monotonicity conditions.
  Heged\H us and Frankl \cite{H24} proved that for a skew Bollob\'{a}s system $\mathcal{P}=\{(A_i,B_i):i\in[m]\}$   of subsets of $[n]$ with $|A_i|=a_i$ and $|B_i|=b_i$ for every $i$, then
    $$
  \sum_{i=1}^{m} \frac{1}{\binom{a_i+b_i}{b_i}}\le n+1.
    $$
This bound is tight but depends on $n$.  It was recently  strengthened by
Yue \cite{Y26}, who showed the following stronger inequality independent of $n$:
    $$
    \sum_{i=1}^{m}\frac{1}{(1+a_i+b_i)\binom{a_i+b_i}{a_i}}\le 1.
    $$
In the same paper, Yue \cite{Y26} also extended the Heged\H us-Frankl  inequality  to the setting of  $r$-partitions.

\begin{theorem}[Yue \cite{Y26}]\label{Y26}
    Suppose that $[n]$ is the disjoint union of some sets $X_1,\ldots,X_r$ with $|X_k|=n_k$.
    For  $A_i,B_i\subseteq[n],\ i\in[m]$, let  $\mathcal{P}=\{(A_i,B_i):i\in[m]\}$ be a skew Bollob\'{a}s system satisfying that
    $$ |A_i\cap X_k|=a_{i,k},~|B_i\cap X_k|=b_{i,k},~~\forall~ i\in[m],~k\in[r].
    $$
    Then
    $$
   \sum_{i=1}^{m}\prod_{k=1}^{r}\frac{1}{\binom{a_{i,k}+b_{i,k}}{a_{i,k}}}\le\prod_{k=1}^{r}(1+n_k)\le \left(1+\frac{n}{r}\right)^r.
    $$
\end{theorem}

Motivated by the above results, we in this paper consider their natural analogues in the vector space setting.
 \begin{definition}
     Let $V$ be  an $n$-dimension real vector space,  and $\mathcal{P}=\{(A_i,B_i):i\in[m]\}$ be a collection of pairs of subspaces of $V$. Then $\mathcal{P}$ is called a  \textit{skew Bollob\'as system} if
     \begin{itemize}
    \item[(i)] for any $\ i\in[m]$, $\dim(A_i\cap B_i)=0$; and

     \item[(ii)]   for any $\ i<j\in[m]$, $\dim(A_i\cap B_j)>0$.
     \end{itemize}
 \end{definition}

In 2021, Scott and Wilmer \cite{S21} first obtained a weighted skew  Bollob\'{a}s-type inequality in the subspace setting. Using the exterior algebra method, they proved that  under the additional monotonicity condition $a_1\leq a_2\leq\cdots \leq a_m$ and $b_1\geq b_2\geq\cdots \geq b_m$, where $\dim (A_i)=a_i$ and $\dim (B_i)=b_i$, the following inequality holds:
$$
\sum_{i=1}^m \frac{1}{
\binom{a_i + b_i}{a_i}} \leq 1 .
$$
Recently, Wu, Li, Lu and Feng \cite{W26} further generalized this result to  a  subspace analogue of $r$-partitions. Their generalization also yields a subspace extension of Alon's theorem  \cite{A851}.
Using the weight invariance argument, they also  gave a subspace extension of Yue's inequality,
showing that, for any skew Bollob\'as system  $\mathcal{P}=\{(A_i,B_i):i\in[m]\}$  of subspaces of $V$ with $\dim (A_i)=a_i$ and $\dim (B_i)=b_i$, one has
$$
\sum_{i=1}^m \frac{1}{\bigl(a_i +b_i+1\bigr)
\binom{a_i + b_i}{a_i}} \leq 1 .
$$

In this paper, we continue this line of investigation by establishing a subspace extension of Yue's second result (Theorem \ref{Y26}), which concerns skew Bollob\'as  systems with respect to a fixed direct sum decomposition.

 Our first  result is:

\begin{theorem}\label{main}
   Let $V$ be an $n$ dimensional real vector space with a direct sum decomposition $V = V_1 \oplus \cdots \oplus V_r$. Suppose that $\mathcal{P}=\{(A_i,B_i):i\in[m]\}$ is a skew Bollob\'{a}s  system of subspaces of $V$ satisfying $
A_i = \bigoplus_{k=1}^r (A_i \cap V_k)$ and $ B_i = \bigoplus_{k=1}^r (B_i \cap V_k)$ for each $i\in [m]$.  Set $a_{i,k}=\dim({A_i\cap V_k})$ and $b_{i,k}=\dim({B_i\cap V_k})$ for $k\in[r]$.
Then we have
    $$
    \sum_{i=1}^{m}\left [\prod_{k=1}^{r}\binom{a_{i,k}+b_{i,k}}{a_{i,k}}(1+a_{i,k}+b_{i,k})\right ]^{-1}
    \le 1.$$
\end{theorem}

Theorem \ref{main} indeed implies the set version in the following way. Let $\mathcal{P} = \{(A_i, B_i) : i \in [m]\}$ be a skew Bollob\'as system of subsets of $[n]$, and let $[n] = X_1 \cup \cdots \cup X_r$ be a partition into pairwise disjoint subsets. Consider the standard basis $\{e_1,\ldots,e_n\}$ of $\mathbb{R}^n$ and define $V_k = \operatorname{span}\{e_p : p \in X_k\}$ for each $k\in[r]$, so that $\mathbb{R}^n = \bigoplus_{k=1}^r V_k$. For each $i\in[m]$, set $A_i' = \operatorname{span}\{e_p : p \in A_i\}$ and $B_i' = \operatorname{span}\{e_p : p \in B_i\}$. Then $\mathcal{P}' = \{(A_i', B_i') : i \in [m]\}$ is a skew Bollob\'as system of subspaces of $\mathbb{R}^n$. For $i\in[m]$ and $k\in[r]$, we have $\dim(A_i' \cap V_k) = |A_i \cap X_k|$ and $\dim(B_i' \cap V_k) = |B_i \cap X_k|$, and clearly $A_i' = \bigoplus_{k=1}^r (A_i' \cap V_k)$ and $B_i' = \bigoplus_{k=1}^r (B_i' \cap V_k)$. Applying Theorem \ref{main} to $\mathcal{P}'$ yields the following corollary, which strengthens Theorem \ref{Y26}.

\begin{corollary}\label{cor1}
Suppose that $[n]$ is the disjoint union of the subsets $X_1,\ldots,X_r$.  Let $\mathcal{P}=\{(A_i,B_i):i\in[m]\}$ be a skew Bollob\'{a}s  system  of subsets of $[n]$ with $a_{i,k}=|A_i\cap X_k|$ and $b_{i,k}=|B_i\cap X_k|$.  Then
    $$
     \sum_{i=1}^{m}\left [\prod_{k=1}^{r}\binom{a_{i,k}+b_{i,k}}{a_{i,k}}(1+a_{i,k}+b_{i,k})\right ]^{-1}
    \le 1.
    $$
\end{corollary}

Apart from the skew setting, there is another natural relaxation of the Bollob\'as condition, introduced by Tuza \cite{T87,T89}. Instead of requiring a single pair $(A_i,B_i)$ with $A_i\cap B_i=\emptyset$ and $A_i\cap B_j\neq\emptyset$ for $i<j$, Tuza considered families of $d$-tuples, and the nonempty intersection condition is required only for some ordered pair of components.

\begin{definition}\label{de7}
    Let $\mathcal{P}=\{(A_i^{(1)},\ldots,A_i^{(d)}):i\in[m]\}$ be
    a collection of $d$-tuples of subsets of $[n]$.
     Then $\mathcal{P}$ is called a \textit{weak Bollob\'as system} if the following hold:
\begin{itemize}
    \item[(i)] for any $~ i\in [m]$, the subsets $A_i^{(1)}, \ldots, A_i^{(d)} $ are pairwise disjoint.
     \item[(ii)] for any $\ i<j\in[m]$,  $\exists\ p<q\in[d]$ such that
     $
  A_i^{(p)}\cap A_j^{(q)} \ne\emptyset\ \text{or} \ A_i^{(q)}\cap A_j^{(p)}\ne\emptyset.
    $
\end{itemize}
\end{definition}

Every skew Bollob\'{a}s system is weak, but the converse does not hold.

For such systems,  using a probabilistic argument based on random partitions of the ground set, Tuza \cite{ T89} obtained the following  result.

\begin{theorem}[Tuza \cite{T89}]\label{T89}
    Let $p_1,\ldots,p_d$ be arbitrary positive real numbers such that $p_1+\ldots +p_d=1$. Suppose that $\mathcal{P}=\{(A_i^{(1)},\ldots,A_i^{(d)}):i\in[m]\}$ is a weak Bollob\'{a}s system of $d$-tuples of subsets of $[n]$. Then
    $$
    \sum_{i=1}^{m}\prod_{\ell=1}^{d} p_r^{|A_i^{(\ell)}|}\le 1.
    $$
\end{theorem}

For the special case $d=2$, we may write $A_i=A_i^{(1)}$, $B_i=A_i^{(2)}$, then $\mathcal{P}=\{(A_i,B_i):i\in[m]\}$ is a weak Bollob\'as system. If $|A_i|=a$ and $|B_i|=b$ for all $i$, in view of Theorem \ref{T89}, Tuza \cite{T87} derived the
%uniform
bound $m\leq \frac{(a+b)^{a+b}}{a^ab^b}$.
Theorem \ref{T89} also has many applications in other various combinatorial problems, including bounds for graph coloring \cite{K87, T75} and hypergraph covering numbers \cite{T}.

  In this paper, we give a new combinatorial proof for Theorem \ref{T89}.
Motivated by this new proof, we consider its natural extension of the vector space setting. We first introduce the appropriate definitions of $d$-tuples for subspace.

\begin{definition}
    Let  $V$ be an $n$ dimensional real vector space, and let $\mathcal{P}=\{(A_i^{(1)},\ldots,A_i^{(d)}):i\in[m]\}$ be a collection of $d$-tuples of subspaces of $V$. Then
   \begin{itemize}
       \item [(1)] $\mathcal{P}$ is called a skew Bollob\'{a}s  system if
       \begin{itemize}
        \item [(i)]   for any $\ i \in [m], \ \dim{(A_i^{(1)}+\cdots+A_i^{(d)})}=\dim{(A_i^{(1)})}+\cdots+\dim{(A_i^{(d)})}$; and
         \item [(ii)] for
        any $\ i< j\in[m],\ \exists\ p<q\in[d]$ such that $\dim{(A_i^{(p)}\cap A_j^{(q)})}>0$.
   \end{itemize}
       \item[(2)] $\mathcal{P}$ is called a weak Bollob\'{a}s  system if
         \begin{itemize}
        \item [(i)] for any $\ i \in [m], \ \dim{(A_i^{(1)}+\cdots+A_i^{(d)})}=\dim{(A_i^{(1)})}+\cdots+\dim{(A_i^{(d)})}$; and
        \item [(ii)] for
        any $\ i< j\in[m],\ \exists\ p<q\in[d]$ such that $\dim{(A_i^{(p)}\cap A_j^{(q)})}>0$ or $\dim{(A_i^{(q)}\cap A_j^{(p)})}>0$.
      \end{itemize}
      \end{itemize}

\end{definition}

The weak condition is the subspace analogue of Tuza's original definition for sets. However, a direct extension of Theorem \ref{T89} to weak Bollob\'{a}s  systems of subspaces encounters an obstacle: a uniform bound analogous to Lemma \ref{Y262} (i) is not available for weak systems. This limitation  leads us to consider the stronger skew  Bollob\'{a}s
condition.

Our second main result is a subspace analogue of Tuza's theorem for skew Bollob\'{a}s  system.

\begin{theorem}\label{main2}
 Let  $V$ be an $n$ dimensional real vector space, and let  $p_1,\ldots,p_d$ be arbitrary positive real numbers with $p_1+\ldots+p_d=1$.
    Suppose that $\mathcal{P}=\{(A_i^{(1)},\ldots,A_i^{(d)}):i\in[m]\}$ is a skew Bollob\'{a}s  system of subspaces of $V$. Let  $a_{i}^{(\ell)}=\dim({A_i^{(\ell)}})$ for $\ell\in[d]$ and $i\in[m]$. Then
    $$
\sum_{i=1}^{m}\prod_{\ell=1}^{d}p_{\ell}^{a_i^{(\ell)}}
    \le 1.$$
\end{theorem}

%As an application, we show that Theorem \ref{main2} implies Theorem \ref{T89}. Let $\mathcal{P}=\{(A_i^{(1)},\ldots,A_i^{(d)}):i\in[m]\}$ be a weak Bollob\'{a}s system of $d$-tuples of subsets of $[n]$. By definition \ref{de7}, each $d$-tuple occurs at most once in $\mathcal{P}$, and since there are $(d+1)^n$ possible $d$-tuples of pairwise disjoint subsets of $[n]$, we have $m\leq (d+1)^n$. Now consider the standard basis $\{e_1,\ldots,e_n\}$ of $\mathbb{R}^n$. For each $i\in[m]$ and $r\in[d]$, define $B_i^{(r)} = \operatorname{span}\{e_p : p \in A_i^{(r)}\}\subseteq\mathbb{R}^n$. Then $\mathcal{P}' = \{(B_i^{(1)},\ldots,B_i^{(d)}):i\in[m]\}$ forms a weak Bollob\'{a}s system of subspaces of $\mathbb{R}^n$ satisfying the conditions of Theorem \ref{main2}. Applying Theorem \ref{main2} to $\mathcal{P}'$ yields Theorem \ref{T89}.

The proofs of Theorems \ref{main}, \ref{T89} and  \ref{main2} rely on the so called weight-invariance technique.
This method is originated from  \cite{H24},  and subsequently refined in \cite{W26}, to deal with  Bollob\'as-type inequalities as well as Tuza-type theorems.
For more related results on Bollob\'as-type theorems and their variations, we refer the reader to \cite{H15, H21, K841, K21, Y22, Y262} and the references therein.

%This paper is organized as follows.

\section{Proofs of Theorems \ref{main},  \ref{T89} and  \ref{main2} }

\subsection{Proof of Theorem \ref{main}}

Our proof relies crucially on a recent result established by Wu et al. \cite{W26}, which provides a general bound for skew Bollob\'as systems of subspaces under direct sum decompositions.

\begin{theorem}[\cite{W26}]\label{adt1}
 Let $V$ be an $n$ dimensional real vector space with a direct sum decomposition $V = V_1 \oplus \cdots \oplus V_r$. Let   $\mathcal{P}=\{(A_i,B_i):i\in[m]\}$ be a skew Bollob\'as system of subspaces of $V$ satisfying $
A_i = \bigoplus_{k=1}^r (A_i \cap V_k)$ and $ B_i = \bigoplus_{k=1}^r (B_i \cap V_k)$ for each $i\in [m]$. Suppose that
$ \dim (A_i \cap V_k)=a_k$ and $ \dim (B_i \cap V_k)=b_k$ for every $i\in [m]$ and $k\in [r]$.
Then
$$m\leq \prod_{k=1}^r {a_k+b_k \choose a_k}.$$
\end{theorem}

\begin{lemma}[\cite{H24}]\label{thm1}
    Let $\mathcal{P}=\{(A_i,B_i):i\in[m]\}$ be a skew Bollob\'{a}s system of subspaces of $V\cong\mathbb{R}^n$.  Then  $m\leq 2^n$.
\end{lemma}

\begin{proof}[\bf Proof of Theorem \ref{main}]
Assume $n_k= \dim(V_k)$.
Since $V = V_1 \oplus \cdots \oplus V_r$, we have $n=n_1+\cdots+n_r$.
For each $k\in [r]$, we denote  $$
D_k=\{(a,b)\in \mathbb{Z}\times \mathbb{Z}: 0\leq a,b \leq n_k, a+b\le n_k\}.
$$
%Let  $\mathcal{P}=\{(A_i,B_i):i\in[m]\}$ be a skew Bollob\'{a}s  system of subspaces of $V$ with  $
%A_i = \bigoplus_{k=1}^r (A_i \cap V_k)$ and $ B_i = \bigoplus_{k=1}^r (B_i \cap V_k)$,
%and set  $a_{i,k}=\dim({A_i\cap V_k})$ and $b_{i,k}=\dim({B_i\cap V_k})$.
With the symbols in the assumption, we
define the \textbf{type} of $\mathcal{P}$ as
$$
\tau(\mathcal{P})=\left(m; (a_{1,1},b_{1,1}),\ldots, (a_{1,r},b_{1,r}), \ldots, (a_{m,1},b_{m,1}),\ldots (a_{m,r},b_{m,r})\right),
$$
and define the \textbf{weight} of $\mathcal{P}$  as
$$
\omega(\mathcal{P})= \sum_{i=1}^{m}\left [\prod_{k=1}^{r}\binom{a_{i,k}+b_{i,k}}{a_{i,k}}(1+a_{i,k}+b_{i,k})\right ]^{-1}.
$$
By Lemma \ref{thm1}, we have $m\le2^n$. Since  $(a_{i,k},b_{i,k})\in D_k$, the number of possible types of $\mathcal{P}$ is bounded by $$\sum_{m=0}^{2^n}(| D_1 |\cdots |D_r|) ^m,$$ which is finite.

If there exists some $\mathcal{P}$ having type $\tau$, then we say that $\tau$ is \textbf{realizable}, and we define $\omega(\tau)=\omega(\mathcal{P})$.
Let $T$ denote the set of all realizable types. Since $T$ is finite, the set $\{\omega(\tau): \tau \in T\}$ is a finite set of real numbers.
Therefore we may choose a  realizable type $\tau^*$  such that $\omega(\tau^*)=\text{max}~\omega(\tau)$.
Without loss of generality, we may assume that the skew Bollob\'{a}s  system $\mathcal{P}=\{(A_i,B_i):i\in[m]\}$ of the theorem satisfies  $\omega(\mathcal{P})=\omega(\tau^*)$.

Fix a pair $(A_i,B_i)\in \mathcal{P}$ and  $k\in [r]$. If there exists a vector $ x\in V_k\setminus ((A_i\cap V_k)\oplus (B_i\cap V_k))$,  then we consider two new pairs
$
(A_i\oplus\left \langle x \right \rangle, B_i )$
 and
$ (A_i,B_i\oplus \left \langle x \right \rangle ).
$
By the definition of a skew  Bollob\'{a}s  system,  neither of these pairs belongs to
$
\mathcal{P}.
$
Replacing $(A_i, B_i) \in \mathcal{P}$ with $(A_i\oplus\left \langle x \right \rangle,B_i ),(A_i,B_i\oplus \left \langle x \right \rangle )$, we obtain a new skew  Bollob\'{a}s  system $\mathcal{P}'$.
Let $s_{i,k}=a_{i,k}+b_{i,k}$. A direct computation shows that
$$
\begin{aligned}
& \omega(\mathcal{P}')-\omega(\mathcal{P})\\
=&\frac{1}{\prod_{j\neq k}\binom{s_{i,j}}{a_{i,j}}(1+s_{i,j})}\left(\frac{1}{\binom{s_{i,k}+1}{a_{i,k}+1}(s_{i,k}+2)}
+\frac{1}{\binom{s_{i,k}+1}{a_{i,k}}(s_{i,k}+2)}-\frac{1}{\binom{s_{i,k}}{a_{i,k}}(1+s_{i,k})}\right)\\
=&0.
\end{aligned}
$$
Thus, $\omega(\mathcal{P})=\omega(\mathcal{P}^\prime)$. This weight invariance  allows us to modify a system by such operations without changing its weight.

For a pair $(A_i,B_i)\in \mathcal{P}$ and a subspace $V_k$,  we denote
%define the \textbf{deficiency}
$$
d_{i,k}:=n_k-\dim\big( (A_i\cap V_k)\oplus ( B_i\cap V_k)\big)\ge 0.
$$
A pair $(A_i,B_i)$ is called \textbf{full} if $d_{i,k}=0$ for all $k$, i.e.,
$
(A_i\cap V_k)\oplus(B_i\cap V_k)=V_k
$
for all $k$.
For a full pair $(A_i,B_i)$, we have $a_{i,k}+b_{i,k}=n_k$ for all $k$, and $A_i\oplus B_i=V$.
We define the \textbf{potential} of  $\mathcal{P}=\{(A_i,B_i):i\in[m]\}$ as
$$
\phi(\mathcal{P})=\sum_{i=1}^{m}\prod_{k=1}^{r}2^{n_k-d_{i,k}}.
$$
If $\mathcal{P}$ contains a non-full pair, then there exist $i\in[m]$ and $k\in[r]$ such that $d_{i,k}\ge 1$.
Choose a vector $ x\in V_k\setminus ((A_i\cap V_k)\oplus (B_i\cap V_k))$, we can construct $\mathcal{P}'$ as described above. Then we similarly have $$\omega(\mathcal{P}')=\omega(\mathcal{P}),
$$
 and
$$
\begin{aligned}
\phi(\mathcal{P}')-\phi(\mathcal{P})
&=2\left(\prod_{j\neq k}2^{n_j-d_{i,j}}\right)2^{n_k-(d_{i,k}-1)}-\left(\prod_{j\neq k}2^{n_j-d_{i,j}}\right)2^{n_k-d_{i,k}}\\
&=3\prod_{j=1}^{r}2^{n_j-d_{i,j}}>0.
\end{aligned}
$$
Thus, the potential strictly increases under such an operation.
By Lemma \ref{thm1}, we have
$$
\phi(\mathcal{P}')\le 2^n\cdot\prod_{k=1}^{r}2^{n_k}= 4^n.
$$
Since $\phi$ is strictly increasing under each such operation and bounded above by $4^n$, this process must stop in finite steps. Hence, repeated applications of this operation must eventually yield a system consisting entirely of full pairs. We may therefore assume without loss of generality that $\mathcal{P} = \{(A_i, B_i) : i \in [m]\}$ is such a system. Consequently,
$
(A_i\cap V_k)\oplus(B_i\cap V_k)=V_k
$
for all $k\in [r]$ and $i\in [m]$.

For a fixed $\mathbf{a}=(a_1,\ldots,a_r)$ with $0\leq a_k\leq n_k$ for each  $k$, we define
$$
I_{\mathbf{a}}= \left \{i\in[m]:(a_{i,1},\ldots,a_{i,r})=(a_1,\ldots,a_r)\right\}.
$$
For any $i\in I_{\mathbf{a}}$ and $k\in[r]$, the fullness implies
$$
\begin{aligned}
    \dim(A_i\cap V_k)=a_k,\
   \dim(B_i\cap V_k)=n_k-a_k.
\end{aligned}
$$
Hence, the subfamily  $\{(A_i,B_i):i\in I_\mathbf{a}\}$ is a  skew Bollob\'{a}s system satisfying the conditions in Theorem \ref{adt1}, so we have
$$
\left | I_{\mathbf{a}} \right |\le\prod_{k=1}^r {a_k+n_k-a_k \choose a_k}=\prod_{k=1}^r {n_k\choose a_k}.
$$
Thus,
$$
\begin{aligned}
\omega(\mathcal{P})&=\sum_{i=1}^{m}\prod_{k=1}^{r}\left[\binom{n_k}{a_{i.k}}(1+n_k)\right]^{-1}\\
&=\left(\prod_{k=1}^{r}\frac{1}{1+n_k}\right)\left(\sum_{i=1}^{m}\prod_{k=1}^{r}\frac{1}{\binom{n_k}{a_{i,k}}}\right)\\
    &=\left(\prod_{k=1}^{r}\frac{1}{1+n_k}\right)\sum_{\mathbf{a}=(a_1,\ldots,a_r)}\sum_{i\in I_{\mathbf{a}}}\frac{1}{\prod_{k=1}^r {n_k\choose a_k}}\\
    &\leq\left(\prod_{k=1}^{r}\frac{1}{1+n_k}\right)\sum_{\mathbf{a}=(a_1,\ldots,a_k)}1\\
    &\le 1.
\end{aligned}
$$
This completes the proof of Theorem \ref{main}.
\end{proof}

\subsection{An alternative proof of Theorem \ref{T89}}

In this subsection, we present a new combinatorial proof of Theorem \ref{T89}.

\begin{proof}[\bf Proof of Theorem \ref{T89}]
Let $\mathcal{P}=\{(A_i^{(1)},\ldots,A_i^{(d)}):i\in[m]\}$ be a weak Bollob\'{a}s system of $d$-tuples of subsets of $[n]$.
By definition, each $d$-tuple occurs at most once in $\mathcal{P}$, and since there are $(d+1)^n$ possible $d$-tuples of pairwise disjoint subsets of $[n]$, we have $m\leq (d+1)^n$. Now we define
$$
\phi(\mathcal{P})=\sum_{i=1}^{m}\sum_{\ell=1}^{d}|A_i^{(\ell)}|.
$$
It follows that $\phi(\mathcal{P})\leq n(d+1)^n$. We define
$$
\omega(\mathcal{P})=\sum_{i=1}^m\prod_{\ell=1}^{d}p_{\ell}^{|A_i^{(\ell)}|}.
$$
    Consider a tuple $(A_i^{(1)},\ldots,A_i^{(d)})\in \mathcal{P}$. If there exists $x\in[n]\setminus\Bigl(\bigcup_{j\in[d]}A_i^{(j)}\Bigr)$,  we can construct $d$ new tuples:
    $$
   \left (A_i^{(1)}\cup\{x\},\ldots,A_i^{(d)}\right),\ldots,   \left (A_i^{(1)},\ldots,A_i^{(d)}\cup\{x\}\right),
    $$
    each is not contained in $\mathcal{P}$. Replacing $(A_i^{(1)},\ldots,A_i^{(d)})$ in  $\mathcal{P}$ with these $d$ new tuples  yields a new weak Bollob\'{a}s  system $\mathcal{P}'$.
    Then
     $$
    \phi(\mathcal{P}')\leq n(d+1)^n.
    $$
     Moreover,
    \begin{align*}
    \omega(\mathcal{P}')-\omega(\mathcal{P})&=\prod_{\ell=1}^{d}p_{\ell}^{|A_i^{(\ell)}|}\big(p_1+\cdots+p_d-1\big)=0,\\
    \phi(\mathcal{P}')-\phi(\mathcal{P})&=d>0.
    \end{align*}
    By  repeated applications of the above replacement operation, we may eventually obtain a weak Bollob\'{a}s  system $\mathcal{P}^*=\{(B_i^{(1)},\ldots,B_i^{(d)}):i\in[m^*]\}$ such that   $\omega(\mathcal{P}^*)=\omega(\mathcal{P})$, and
    $$
   B_i^{(1)}\cup\cdots\cup B_i^{(d)}=[n]~ \text{ for all } ~i\in [m^*].
    $$
    For a fixed $\mathbf{a}=\{a_1,\ldots,a_d\}$ with  $a_1+\cdots +a_d= n$, we define
    $$
    I_{\mathbf{a}}=\left\{i\in[m^*]:(|B_i^{(1)}|,\ldots,|B_i^{(d)}|)=(a_1,\ldots,a_d)\right\}.
    $$
    Note that, here,  for  $n= a_1+\cdots+a_d$, the multinomial coefficient is written as  $\binom{n}{a_1,\ldots,a_d}:=\frac{n!}{a_1!\cdots a_d!}$. Then
    $$
    |I_{\mathbf{a}}|\le\binom{n}{a_1,\ldots,a_d}.
    $$
Consequently,
$$
\begin{aligned}
    \omega(\mathcal{P})&= \omega(\mathcal{P}^*)\\
    &=\sum_{\mathbf{a}=(a_1,\ldots,a_d)}\sum_{i\in I_{\mathbf{a}}}p_1^{a_1}\ldots p_d^{a_d}\\
    &\le\sum_{\mathbf{a}=(a_1,\ldots,a_d)}\binom{n}{a_1,\ldots,a_d}p_1^{a_1}\cdots p_d^{a_d}\\
    &=(p_1+\cdots +p_d)^n=1.
\end{aligned}
$$
This yields the desired result.
\end{proof}

\subsection{Proof of Theorem \ref{main2}}

We begin by recalling a result on skew Bollob\'as systems, which can be found in \cite{E08, O18, W26, Y262}.

\begin{lemma}\label{Y262}
    Let $\mathcal{P}=\{(A_i^{(1)},\ldots,A_i^{(d)}): i\in[m]\}$ be a skew  Bollob\'{a}s system of $d$-tuples of subspaces of a vector space $V\cong\mathbb{R}^n$. Then the following hold:
     \begin{itemize}
    \item[(i)] If  $\dim{(A_i^{(\ell)})}=a_{\ell}$ for every $i\in[m]$ and $\ell\in[d]$, then
    $
    m\le\binom{a_1+\cdots+a_d}{a_1,\ldots,a_d}.
    $
     \item[(ii)] In general, $m\leq d^n$.
\end{itemize}
\end{lemma}

\begin{proof}[\bf Proof of Theorem \ref{main2}]
Let $\mathcal{P}=\{(A_i^{(1)},\ldots,A_i^{(d)}):i\in[m]\}$ be a skew Bollob\'{a}s  system of subspaces of $V$ with $a_{i}^{(\ell)}=\dim({A_i^{(\ell)}})$.
By Lemma \ref{Y262} (ii), we have  $m\leq d^n$.

Define
$$
\omega(\mathcal{P})=\sum_{i=1}^{m}\prod_{\ell=1}^{d}p_{\ell}^{a_i^{(\ell)}}.
$$
As in the proof of Theorem \ref{main}, we may assume that $
\omega(\mathcal{P})$ is maximal.

Define
$$
\phi(\mathcal{P})=\sum_{i=1}^{m}\sum_{\ell=1}^{d}a_{i}^{(\ell)}
$$
and note that $\phi(\mathcal{P})\leq nd^n$.
We may further assume, by the argument used in the proof of  Theorem \ref{T89}, that
 $\mathcal{P}=\{(A_i^{(1)},\ldots,A_i^{(d)}):i\in[m]\}$  satisfies
    $$
   A_i^{(1)}\oplus\cdots\oplus A_i^{(d)}=V~ \text{ for all } ~i\in [m].
    $$
 For a fixed $\mathbf{a}=\{a_1,\ldots,a_d\}$ with  $a_1+\cdots +a_d= n$, we define
    $$
    I_{\mathbf{a}}=\left\{i\in[m] : (a_{i}^{(1)},\ldots,a_{i}^{(d)})=(a_1,\ldots,a_d)\right\}.
    $$
Since the subfamily $\{(A_i^{(1)},\ldots,A_i^{(d)}):i\in I_\mathbf{a}\}$ is a  skew Bollob\'{a}s system,  applying Lemma \ref{Y262} (i), it yields
$$
\left | I_{\mathbf{a}} \right |\le\binom{a_1+\cdots+a_d}{a_1,\ldots,a_d}=\binom{n}{a_1,\ldots,a_d}.
$$
Consequently,
$$
\begin{aligned}
    \omega(\mathcal{P})&=\sum_{i=1}^{m}p_1^{a_i^{(1)}}\ldots p_d^{a_i^{(d)}}\\
    &=\sum_{\mathbf{a}=(a_1,\ldots,a_d)}\sum_{i\in I_{\mathbf{a}}}p_1^{a_1}\ldots p_d^{a_d}\\
    &\le\sum_{\mathbf{a}=(a_1,\ldots,a_d)}\binom{n}{a_1,\ldots,a_d} p_1^{a_1}\ldots p_d^{a_d}\\
    &=(p_1+\ldots p_d)^n=1.
\end{aligned}
$$
This completes the proof.
\end{proof}

\section{Concluding Remarks}
A natural question arising from our work concerns the weak Bollob\'{a}s condition for subspaces. In the set setting, Tuza's theorem holds for weak Bollob\'as systems, where the nonempty intersection condition is required only for some ordered pair of components. However, as noted in Section \ref{se1}, a direct subspace analogue of the weak condition does not admit a uniform bound analogous to Lemma \ref{Y262} (i), which is essential for the weight invariance argument.
We propose the following open problem regarding the weak Bollob\'{a}s  system of subspaces.

\begin{problem}
    Let  $V$ be an $n$ dimensional real vector space, and let  $p_1,\ldots,p_d$ be arbitrary positive real numbers with $p_1+\ldots+p_d=1$.
    Suppose that $\mathcal{P}=\{(A_i^{(1)},\ldots,A_i^{(d)}):i\in[m]\}$ is a weak Bollob\'{a}s  system of subspaces of $V$. Let  $a_{i}^{(\ell)}=\dim({A_i^{(\ell)}})$ for $\ell\in[d]$ and $i\in[m]$. Then does the inequality
    $$
\sum_{i=1}^{m}\prod_{\ell=1}^{d}p_{\ell}^{a_i^{(\ell)}}
    \le 1
    $$
    holds? Or whether additional structural assumptions are necessary?
\end{problem}

%Another natural direction for further research is to consider skew Bollob��s systems of $d-$tuples equipped with a direct sum decomposition of space. The following conjecture would serve as a common generalization of Theorems \ref{main}.
%\begin{conjecture}
%       Let $V_1,\ldots,V_r$ be disjoint subspaces of $V\cong\mathbb{R}^n$, for which $V_1\oplus\ldots\oplus V_r=V$. Suupose that $\mathcal{P}=\{(A_i^{(1)},\ldots,A_i^{(d)}):i\in[m]\}$ is a skew Bollob\'{a}s system of $d$-tuples of $V$. Let $\dim{A_i^{(e)}\cap V_k}=a_{i,k}^e$ and $s_k=\dim{(\oplus_{i=1}^m\oplus_{e=1}^dA_i^{(e)})\cap V_k}$ for $k\in[r]$. Then
%    $$
%    \sum_{i=1}^{m}\left [\prod_{k=1}^{r}\binom{\sum_{e=1}^da_{i,k}^e}{a_{i,k}^1,\ldots,a_{i,k}^d}\right ]^{-1}\le\prod_{k=1}^r\binom{s_k+d-1}{d-1}.
%    $$
%\end{conjecture}
%If true, this would unify Theorem \ref{main} (the case $d=2$), and would serve as a common extension.

\section*{Declaration of competing interest}We declare that we have no conflict of interest to this work.

\section*{Data availability}No data was used for the research described in the article.

% \section*{Acknowledgement}
%The authors would like to express their sincere thanks to the referee for the valuable suggestions which greatly improved the presentation of the %manuscript.

\end{document}